\documentclass[11pt]{article}
\usepackage{amsmath,amssymb,amsthm,amscd,dsfont}
\usepackage{amsfonts,latexsym,rawfonts,amsmath,amssymb,amsthm}
\usepackage{lscape}
\usepackage{amscd, float,times,rotating}
\usepackage{pb-diagram}
\usepackage{hyperref}

\numberwithin{equation}{section}

\def\g{{\bar{g}}}
\def\M{{\partial M}}
\def\W{{\bar{W}}}
\def\R{{\bar{R}}}
\def\A{{\bar{A}}}
\def\D{{\bar{D}}}
\def\S{{\bar{S}}}
\def\H{{\bar{H}}}

\newtheorem{prop}{Proposition}[section]
\newtheorem{theo}[prop]{Theorem}
\newtheorem{lemm}[prop]{Lemma}

\newtheorem{rema}[prop]{Remark}

\def\begeq{\begin{equation}}
\def\endeq{\end{equation}}

\def\lab{\ }
\def\lab{\label}

\begin{document}
\title{Boundary Regularity for Asymptotically Hyperbolic Metrics with Smooth Weyl Curvature}
\author{Xiaoshang Jin \thanks{The author's research is partially supported by China Scholarship Council (File
No.201706190079)}}
\date{}
\maketitle
\begin{abstract}
In this paper, we study the regularity of asymptotically hyperbolic metrics with Einstein condition near boundary and Weyl curvature smooth enough in
arbitrary dimension. Following Michael Anderson's method, we show that $C^{m,\alpha}$ conformally compact Riemannian metrics with Einstein equation vanishing to finite order near boundary have conformal compactifications that are $C^{m+2,\alpha}$ up to the boundary when Weyl curvature is in $C^{m,\alpha}$ and the boundary metric is in  $C^{m+2,\alpha}$
where $m\geq 3.$
\end{abstract}

\section{Introduction}  \lab{sect1}
\par It is well known that there are very close connections between the hyperbolic space $\mathbb{H}^{n+1}$ and its boundary, which we know as a sphere $\mathbb{S}^n.$ In recent years, many mathematicians are more interested in conformally compact Einstein manifolds with negative scalar curvature instead of hyperbolic space. The physics community has also become interested in the compact Einstein manifolds since the introduction of the AdS/CFT correspondence proposed by Maldacena in the theory of quantum gravity in theoretic physics. In this paper, we mainly discuss the boundary regularity problem when the Weyl curvature of the compactification has some regularity.
\par Let $M$ be the interior of a compact $(n+1)$-dimensional manifold $\overline{M}$ with non-empty boundary $\M$. We call a complete metric $g_+$ on $M$ is $C^{m,\alpha}$(or $W^{k,p}$) conformally compact if there exits a defining function $\rho$ on $\overline{M}$ such that the conformally equivalent metric $$g=\rho^2g_+$$ can extend to a $C^{m,\alpha}$(or $W^{k,p}$) Riemannian metric on $\overline{M}.$ The defining function is smooth on $\overline{M}$ and satisfies
 \begin{equation}\lab{1.1}
 \left\{
    \begin{array}{l}
    \rho>0\ \ in \ M
    \\\rho=0\ \ on\ \M
    \\d\rho\neq 0\ \ on\  \M
    \end{array}
 \right.
 \end{equation}
 Here $C^{m,\alpha}$ and $W^{k,p}$ are usual H\"older space and the Sobolev space.
\par The induced metric $h=g|_{\M}$ is called the boundary metric associated to the compactification $g.$  The defining function is unique up to a multiplication by a positive function on $\overline{M}$. So the conformal class $[g]$ is uniquely determined by $g_+,$ and the conformal class $[h]$ is uniquely determined by $(M,g_+)$. We call $[h]$ the conformal infinity of $g_+.$ We are interested in Einstein manifolds, which means the metric $g_+$ also satisfies
 \begin{equation}\lab{1.2}
Ric_{g_+}=-ng_+.
 \end{equation}
\par The boundary regularity problem was first raised by Fefferman and Graham in 1985. Namely, given a conformally compact Einstein manifold $(M,g_+)$ and its compactification $g,$ if the boundary metric $h$ is in $C^{m,\alpha},$ is there a $C^{m,\alpha}$ compactification of $g_+$? In fact, Fefferman and Graham noticed that if $dim M=n+1$ is odd, the boundary regularity in general breaks down at the order $n.$ When $dim M=n+1$ is even, the  $C^{m,\alpha}$ compactification should exit. In \cite{A03} and \cite{A08}, M.T.Anderson solved the problem in dimension 4 by using the Bach equation in dimension 4. He only assumed the original compactification $g$ is in $W^{2,p}$ for some $p>4.$ I'm not sure whether the $W^{2,p}$ condition is good enough for the manifold to improve the boundary regularity.  In \cite{Hell}, Helliwell solved the issue in all even dimensions by following Anderson's method. He considered the Fefferman-Graham ambient obstruction tensor instead of Bach tensor in higher dimensions. Helliwell assumed the compactification $g$ is at least in $C^{n,\alpha}$ for a $(n+1)$ smooth manifold. It means the original compactification is  $C^{3,\alpha}$ for a smooth manifold of dimension 4.
\par In this paper, we follow Anderson's approach to study the boundary regularity in general dimensions.  As \cite{FG} pointed that when $dim M=n+1$ is odd, there are log terms in the asymptotic expansion of $g$ near $\M$ at the order $n$. If we add a condition of $g,$ the log term may be ruled out. By studying the equation (\ref{3.8}), we find that if the Weyl curvature and scalar curvature of $g$ is smooth enough near boundary, the log term may not exist. In \cite{J} and \cite{ST}, the equation (\ref{2.5}) tells us that in harmonic coordinates, the regularity of a metric can be improved to two orders higher than the regularity of its Weyl curvature locally in the sense of conformal transformation. The main idea of this paper is to extend the result to the manifolds with boundary.
\par In the end of this paper, we prove that the regularity of defining function is the same as the new structure on $\overline{M}$ for Einstein case. It extends  Helliwell's result in \cite{Hell}, where he obtained that the the regularity of defining function is the same as the original compactification.
\par The equation (\ref{2.5}) holds for all manifolds, not necessary Einstein manifolds. We don't use the Einstein equation in the interior of $\overline{M},$ so we focus on metrics  that satisfy the condition of that Einstein equation vanishes to finite order near boundary, that is
\begin{equation}\lab{1.3}
Ric_{g_+}+ng_+=o(\rho^2).
 \end{equation}
The main result is as follows:
\begin{theo}\lab{theo1.1}
Let $(M,g_+)$ be a conformally compact $(n+1)$-manifold with a $C^{m,\alpha}$ conformal compactification $g=\rho^2g_+$ in a given $C^\infty$ atlas $\{y^\beta\}_{\beta=0}^n$ of $\overline{M}$ near $\M$ ($m\geq 3,0<\alpha<1$). $Ric_{g_+}+ng_+=o(\rho^2).$ $\rho$ is a $C^\infty$ defining function of $y^\beta.$
 If the boundary metric $h=g|_{\partial M}\in C^{m+2,\alpha}(\partial M)$ and the Weyl curvature $W$ of $g$ is in
$C^{m,\alpha}(\overline{M})$ in the atlas $y^\beta,$ then there exits  atlas $\{x^\beta\}$ of $\overline{M}$ near $\M$ and in the atlas $\{x^\beta\},$ $g_+$ has a $C^{m+2,\alpha}$ compactification $\tilde{g}=\tilde{\rho}^2g_+$ with boundary metric $h$. The atlas $\{x^\beta\}$ form a $C^{m+3,\alpha}$ structure of $\overline{M}.$ Further more, if $g_+$ is Einstein, $\tilde{\rho}$ is a $C^{m+2,\alpha} $ function in x-coordinates.
\end{theo}
\begin{rema}\lab{rema1.2}
If the $\dim M=n+1$ is even, $(M,g_+)$ is Einstein and $m+2\geq n,$ then the defining function in theorem(\ref{theo1.1}) $\tilde{\rho}$ is a $C^{m+3,\alpha} $ function in x-coordinates.
\end{rema}
\par In \cite{CDLS}, Chru\'sciel, Delay, Lee and Skinner showed a good result of the boundary regularity of conformal compact Einstein manifolds. They proved that when the boundary metrics are smooth, the $C^2$ conformally compact Einstein metrics have conformal compactifications that are smooth up to the boundary in the sense of $C^{1,\lambda}$ diffeomorphism in dimension 3 and all even dimensions, and polyhomogeneous in odd dimensions greater than 3. The $C^2$ condition is of course weaker than the $C^{n,\alpha}$ condition in Helliwell's paper. I think the $C^2$ condition should be the sharp condition. However, their result only holds for the smooth case. It is unknown whether their method can be used for proving the finite boundary regularity.  In this paper, by assuming a condition of Weyl tensor, we solve the finite regularity problem for a conformal compact manifold, which need not to be Einstein, only need to satisfy (\ref{1.3}). Besides, by observing a calculation in section 3.2, we find that the log term of the formal power series of Weyl tensor vanishes if and only if the obstruction tensor of the metric vanishes. So if we assume the Weyl tensor are in $C^{n-2}$ in Theorem A in \cite{CDLS} when $n+1$ is odd and greater than 3, we can obtain an extended smooth result. That is:
\begin{rema}\lab{rema1.3}
Let $(M,g_+)$ be a conformally compact Einstein $(n+1)$-manifold with a $C^2$ conformal compactification $g=\rho^2g_+.$ If the boundary metric $h=g|_{\partial M}$ is smooth,  then for any $\lambda>0,$ there exists $R>0$ and a $C^{1,\lambda}$ diffeomorphism $\Phi:\M_R\rightarrow \M$ such that
$$\Phi^*g_+=\rho^{-2}(d\rho^2+G(\rho)).$$
Where $\M_R=\partial M\times [0,R],$ $\{G(\rho): 0<\rho\leq R\}$ is a one-parameter family of smooth Riemannian
metrics on $\partial M.$
\par If dim M is even or equal to 3, then $\Phi^*g_+$ is conformally compact of class $C^\infty.$
 \par If dim M is odd and greater than 3, the Weyl tensor of $\rho^2\Phi^*g_+$ are of class $C^{n-2},$ then $\Phi^*g_+$ is conformally compact of class $C^\infty.$
\end{rema}

\par The outline of this paper is as follows. In section 2,we introduce the constant scalar curvature compactification and construct a kind of harmonic coordinates near $\M.$ The regularity of the metric isn't changed in the above two steps. We also study the relationship between Ricci curvature and Weyl curvature when the scalar curvature is a constant.
\par In section 3, we review some background for studying conformally compact Einstein manifolds and  asymptotically hyperbolic metrics satisfying (\ref{1.3}), including the change of curvature under conformal transformation, the existence and regularity of geodesic defining function. We also study the reason why the boundary regularity in general breaks down at order $n$ when $dim M=n+1$ is odd. Then after some simple calculations, we show that the  Weyl curvature has an influence on the regularity in geodesic coordinates. Besides, for conformally compact Einstein manifolds, by calculating the formal power series of Weyl tensor in geodesic coordinates, we show that the obstruction tensor of the conformal metric vanishes if and only if the formal power series of Weyl tensor
doesn't contain $x^{n-2}\log x$ term. Which improves Theorem A in \cite{CDLS}.
\par In section 4, we study some boundary conditions, including the Dirichlet condition for $g_{ij}$ and Ricci curvature, the Neumann condition for $g^{0\beta}.$ We use the geodesic defining function as a transition tool to calculate the Dirichlet condition for Ricci curvature. The regularity would drop one order when we change the defining function to geodesic. The geodesic compactification should be at least $C^2$ so that we can calculate some curvature tensor. That's why we need $m\geq 3$ in Theorem \ref{theo1.1}. We use the property of harmonic coordinates to obtain the Neumann condition for $g^{0\beta}.$
\par In section 5, we use the theory of elliptic system to prove Theorem \ref{theo1.1}. We improve the regularity of the conformal metric and defining function in the new coordinates.
\section{Basic geometry equations in harmonic coordinates}  \lab{sect2}
\par In this section, we discuss some basic geometry equations for the manifold $(\overline{M},g).$  Before doing it, we need to make an appropriate choice of conformal compactification to let the scalar curvature be a constant near boundary and construct harmonic coordinates for the metric. These work can be done on an arbitrary  manifold with boundary.
\subsection{Constant scalar curvature compactification}
\begin{lemm}
Let $(M,g_+)$ be a conformally compact n-manifold, $M$ has a $C^{2,\alpha}$ conformal compactification $g=\rho^2g_+.$  $h=g|_{\M}$ is the boundary metric. Then there exits a $C^{2,\alpha}$ constant scalar curvature compactification $\hat{g}=\hat{\rho}^2g_+$ with boundary metric $h.$
\end{lemm}
\begin{proof}
We only need to solve a Yamabe problem with Dirichlet data. Let $\hat{g}=u^{\frac{4}{n-2}}g$, then we consider the equation
 \begin{equation}\lab{2.1}
 \left\{
    \begin{array}{l}
    \Delta_gu-\frac{n-2}{4(n-1)}Su+\frac{n-2}{4(n-1)}\lambda u^{\frac{n+2}{n-2}}=0
    \\ u>0\ \ in\ \overline{M}
    \\ u\equiv 1\ \ on\  \M
    \end{array}
 \right.
 \end{equation}
 When we choose $\lambda=-1,$ From \cite{Ma} we know the equation always has a $C^{2,\alpha}$ solution. So $\hat{g}=u^{\frac{4}{n-2}}g$ is also in $C^{2,\alpha}.$ Since $ u\equiv 1$ on $\M,$ the boundary metric $h$ is not changed.
\end{proof}
\par From the standard theory for elliptic equations, if $g$ is in $C^{m,\alpha}(\overline{M})$ for some $\alpha\in (0,1),$ then $\hat{\rho}=u^{\frac{2}{n-2}}\rho$ is also in $C^{m,\alpha},$ and $\hat{g}$ is also in $C^{m,\alpha}.$ The Weyl tensor  $\hat{W}=u^{\frac{4}{n-2}}W$ of $\hat{g}$ are in $C^{m,\alpha}.$ So in the following, we don't distinguish $g$ and $\hat{g}.$ When we refer to the compactification $g,$ we mean the scalar curvature $S$ of $g$ is a constant.
\subsection{The harmonic coordinates near boundary}
\par The coordinates $\{x^\beta\}_{\beta=0}^n$ are called harmonic coordinates with respect to $g$ when $\Delta_gx^\beta=0$ for $0\leq\beta\leq n.$ We are now going to construct harmonic coordinate near $\M$ which is also harmonic when restricted on $(\M,h).$ In the following, if there are no special instructions, any use of indices will
follow the convention that Roman indices will range from 1 to n, while Greek indices range from 0 to n.
\par Firstly, for any point $p\in\M,$ there are smooth structure $\{y^\beta\}.$ It is easy to construct new coordinates $\{x^i\}$ on $\ M.$
 $\{x^i\}$ are harmonic coordinates on $(\M,h).$ When $h$ is in $C^{m+2,\alpha}$, $x^i$ are $C^{m+3,\alpha}(\M)$ functions of $y^\beta.$ Then $$h_{ij}=h(\frac{\partial}{\partial x^i},\frac{\partial}{\partial x^j})\in C^{m+2,\alpha}(\partial M).$$
 \par Then by solving a local Dirichlet problem: $\Delta_gx^i=0,$ with the boundary condition $x^i$ as above,we can extend $x^i$ to $\overline{M}.$ Similarly, we can choose a harmonic defining function $x^0$ satisfies $\Delta_gx^0=0,x^0|_{\M}=0.$ It is easy to see $\{x^\beta\}_{\beta=0}^n$ form harmonic coordinates with respect to $g$ in a neighborhood of $\M.$ When $g$ is in $C^{m,\alpha},$ $x^\beta$ are $C^{m+1,\alpha}$ functions of $y^\beta.$ Then $$g_{\alpha\beta}=g(\frac{\partial}{\partial x^\alpha},\frac{\partial}{\partial x^\beta})\in C^{m,\alpha}(\overline{M}).$$
\par So when we change the coordinates $\{y^\beta\}$ to $\{x^\beta\}$, the regularities of $g$ and Weyl tensor on $\overline{M}$ and $h$ on $\M$ are unchanged.
\subsection{Basic geometry equations}
\par In this section, we mainly study the relationship between Weyl tensor and Ricci tensor. For any $(n+1)$-manifold, $0\leq i,j,k,l,h,m\leq n.$ We have
\begin{equation}\lab{2.2}
R_{ijkl}=\frac{1}{n-1}(R_{ik}g_{jl}+R_{jl}g_{ik}-R_{jk}g_{il}-R_{il}g_{jk})-\frac{S}{n(n-1)}(g_{ik}g_{jl}-g_{il}g_{jk})+W_{ijkl}.
\end{equation}
On the other hand, we have the second Bianchi identity for $g.$
$$0=Bian(g,Ric)=g^{jh}(R_{jk,h}-\frac{1}{2}R_{jh,k})$$
When the scalar curvature $S$ is constant ,we get
\begin{equation}\lab{2.3}
g^{jh}R_{jk,h}=0.
\end{equation}
From (\ref{2.2}) and (\ref{2.3}), we have
$$R_{ik,l}-R_{il,k}=\frac{n-1}{n-2}g^{jh}W_{ijkl,h}.$$
It follows $$g^{lt}R_{ik,lt}-g^{lt}R_{il,kt}=\frac{n-1}{n-2}g^{lt}g^{jh}W_{ijkl,ht}.$$
As $$R_{il,kt}=R_{il,tk}+Ric*Rm,\ \ g^{lt}R_{il,tk}=(g^{lt}R_{il,t})_k=0.$$
Where $Ric*Rm$ refers to a bilinear form of $R_{ij}$ and $R_{ijkl}.$
We finally get
\begin{equation}\lab{2.4}
g^{lt}R_{ik,lt}=g*Ric*Rm+g*g*(\partial^2W+\Gamma*\partial W+\Gamma*\Gamma*W+\partial\Gamma*W).
\end{equation}
When $g\in C^{m,\alpha},$ in harmonic coordinates$\{x^\beta\}$, the above equation can be written as
\begin{equation}\lab{2.5}
g^{\gamma\tau}\frac{\partial}{\partial x^\gamma}\frac{\partial}{\partial x^\tau}R_{\alpha\beta}=Q(g,\partial g,\partial^2g,W,\partial W,\partial^2W).
\end{equation}
Where $Q$ is a polynomial. Even when $m=3,$ we can define the first and second covariant derivatives of curvature in the sense of distribution (see \cite{J} and \cite{ST}).
It can be
shown (\ref{2.5}) still holds.
\section{The conformal infinity of asymptotically hyperbolic metrics }
\par We will discuss some background material for conformally compact metrics in this section. As the definition in introduction, Let $\rho$ be a defining function for $M,$ set
$$g=\rho^2g_+.$$
We assume $\rho$ is $C^{m,\alpha}$ on $\overline{M}$ in the initial atlas $\{y^\beta\},$ $m,\alpha$ are defined as in Theorem 1.1. $g_+$ satisfies
$$Ric_{g_+}+ng_+=o(\rho^2),$$
the curvature of $g$ can be expressed as following formulas:
\begin{equation}\lab{3.1}
K_{ab}=\frac{K_{+ab}+|\nabla\rho|^2}{\rho^2}-\frac{1}{\rho}[D^2\rho(e_a,e_a)+D^2\rho(e_b,e_b)].
\end{equation}
\begin{equation}\lab{3.2}
Ric=-(n-1)\frac{D^2\rho}{\rho}+[\frac{n(|\nabla\rho|^2-1)}{\rho^2}-\frac{\Delta\rho}{\rho}]g+o(1),
\end{equation}
\begin{equation}\lab{3.3}
S=-2n\frac{\Delta\rho}{\rho}+n(n+1)\frac{|\nabla\rho|^2-1}{\rho^2}+o(1).
\end{equation}
 Here $D^2$ is the Hessian. It is easy to see that when $g$ is at least $C^2$ in $\overline{M},$  $|\nabla\rho|\rightarrow 1$ as $\rho\rightarrow 0,$
and $|K_{+ab}+1|=O(\rho^2).$ Hence a $C^2$ conformally compact Einstein manifold is asymptotically hyperbolic. At $\M,$ we have $|\nabla\rho|=1,$
Let $D^2\rho|_{\M}=A,$ where A is the second fundamental form of $\M$ in $(\overline{M},g).$
\subsection{Geodesic conformal compactification}
As noted in introduction, defining functions are not unique, but differ by multiplication by positive functions which equal 1 when restricted on $\M.$ When the defining function $r$ and its compactification $\bar{g}=r^2g_+$ satisfies $$|\bar{\nabla} r|_{\bar{g}}\equiv1$$ in a neighborhood of $\M,$ we call $r$ geodesic defining function. We show that defining function always exists.
\begin{lemm}
Let $g$ be a $C^2$ conformal compactification of $(M,g_+),$ $g=\rho^2g_+.$ $h=g|_{\M}$ is the boundary metric. Then $g_+$ has a unique geodesic conformal compactification with the same boundary metric $h.$
\end{lemm}
\begin{proof}
Let $r=u\rho,$ $\bar{g}=r^2g_+.$  The lemma is equivalent to the equation:
  \begin{equation}\lab{3.4}
 \left\{
    \begin{array}{l}
    2(\nabla\rho)(\log u)+\rho|\nabla \log u|_g^2=\frac{1-|\nabla\rho|_g^2}{\rho}
    \\ u\equiv 1 \ on \ \M.
    \end{array}
 \right.
 \end{equation}
By general theory of first order partial differential equations, we know that it has a unique positive solution in a collar neighborhood $U$ of $\ \M.$
\end{proof}
Further more, When $g$ is in $C^{m,\alpha},$ we know that the right hand of the equation is in $C^{m-1,\alpha}$ on the boundary. So we only have $u\in C^{m-1,\alpha},$ which means $\bar{g}$ is a $C^{m-1,\alpha}$ conformal compactification.
\par It is easy to see that when $Ric_{g_+}+ng_+=o(\rho^2)$ for the defining function $\rho,$ we also have $Ric_{g_+}+ng_+=o(r^2)$ for the geodesic defining function.
\\
\par In the following of this section, we assume $g_+$ is Einstein, i.e. $$Ric_{g_+}+ng_+=0.$$
Then the term $o(1)$ in (\ref{3.2}) and (\ref{3.3}) can be removed.\\
When $M$ has a $C^2$ geodesic conformal compactification $\bar{g}$, it is very convenient for us to do some calculation. From (\ref{3.2}) and (\ref{3.3}), we know that the second fundamental form $\bar{A}$ given by $\bar{A}=\bar{D}^2r$ vanishes on $\M.$ By Gauss lemma, $\g$ can split in $U,$
$$\g=dr^2+g_r$$
for a 1-parameter family $g_r$ of metrics on $\M.$
\par Now we choose the local coordinates $(r,x^1,x^2,\cdots,x^n)$ on $\M$ to study the regularity of $\g$ near the boundary $\M.$ Using the equation (\ref{3.2}) and Gauss-Codazzi equation, we finally get
\begin{equation}\lab{3.5}
r\partial_r^2\g_{ij}-(n-1)\partial_r\g_{ij}-\g^{kl}\partial_r\g_{kl}\g_{ij}-r\g^{kl}\partial_r\g_{kl}\partial_r\g_{ij}+
\frac{r}{2}\g^{kl}\partial_r\g_{ik}\partial_r\g_{jl}-2r\bar{R}ic(g_r)_{ij}=0.
\end{equation}
Here $\g_{ij}$ denotes the tensor $g_r$ on $M$ and $\bar{R}ic(g_r)_{ij}$ denotes the Ricci tensor for the induced metric on level sets of $r.$ We assume that $\g$ is smooth enough so that we could calculate its expansion from (\ref{3.5}). Let $r=0,$ we can derive that $\partial_r{\g}(0)=0.$ By using mathematical induction, we differentiate (\ref{3.5}) $p-1$ times with respect to $r$
\begin{equation}\lab{3.6}
\begin{aligned}
r\partial_r^{p+1}\g_{ij}+(p-n)\partial_r^p\g_{ij}-\g^{kl}(\partial_r^p\g_{kl})\g_{ij}=Q_1(\g^{-1},\partial_r^q\g,\partial_r^q\g\partial_s^2\g)
\\ +rQ_2(\g^{-1},\partial_r^q\g, \partial_r^p\g,\partial_r^{p-2}\partial_s^2\g).
\end{aligned}
\end{equation}
Here $q<p,$ $\partial_s\g$ is the differential of $\g$ with respect to $x^i (1\leq i\leq n).$ $Q_1, Q_2$ are the third polynomials.
\par Setting $r=0,$ we can calculate that $\partial_r^p\g(0)=0$ when $p$ is odd and  $\partial_r^p\g(0)$ is uniquely determined by each step when $p$ is even. However, this will break down when $p=n$ if $n+1$ is odd. In that case we only have $\g^{kl}(\partial_r^n\g_{kl})=0$ at $r=0.$ This give no further information at this order.
\par Now we know that $\partial_r\g(0)=\partial_r^3\g(0)=\cdots=\partial_r^{n-1}\g(0)=0,$  that is
$$\partial_r^{n-1}\g(r)=O(1).$$
 Considering (\ref{3.6}) when $p=n-1,$we have
\begin{equation}\lab{3.7}
\partial_r^n\g_{ij}=\frac{\partial_r^{n-1}\g_{ij}-\g^{kl}(\partial_r^{n-1}\g_{kl})\g_{ij}+Q_1}{r}+Q_2.
\end{equation}
If $\partial_r^{n-1}\g(r)\neq O(r),$ there will be the log term in the expansion for $g_r.$  $\partial_r^{n-1}\g(r)=O(r)$ is the necessary condition
to ensure $\g$ $n$-th differentiable.
\subsection{Weyl tensor in Geodesic coordinates}
 Now we show that if the Weyl tensor $\W$ of $\g$ is $(n-2)$-th differentiable, then $\partial_r^{n-1}\g(r)=O(r)$ holds.
\par We begin with (\ref{2.4}) by taking $x^l=r,1\leq i,k\leq n.$
\par In the coordinates $(r,x^1,x^2,\cdots,x^n),$ we have
$$\R_{ik,r}=\partial_r^3\g+\partial_r\g\ast\partial_r^2\g\ast\g+\partial_r\g\ast\partial_r\g\ast\partial_r\g\ast\g+\partial_rP+\partial_r\g\ast P,$$
$$\R_{ir,k}=\partial_rP\ast P+\partial_r\g\ast\partial_r^2\g\ast\g+\partial_r\g\ast\partial_r\g\ast\partial_r\g\ast\g+\partial_r\g\ast\g\ast\R ic.$$
Here $P=P(\g,\partial_{x^i}\g,\partial_{x^i}^2\g)$ is a polynomial for $1\leq i\leq n$.
So from (\ref{2.4})we have:
\begin{equation}\lab{3.8}
\partial_r^3\g=\partial_r\g\ast\partial_r^2\g\ast\g+\partial_r\g\ast\partial_r\g\ast\partial_r\g\ast\g+\partial_rP\ast P+\partial_r\g\ast\g\ast\R ic+\g\ast \nabla \W.
\end{equation}
We already know that $\partial_rg=\partial_r^3g=\cdots\partial_r^{n-3}=O(r).$ Differentiating (\ref{3.8}) $n-4$ times with respect to $r,$ Each term in the right hand is $O(r)$ besides $\g\ast \nabla \W.$ $\partial_r^{n-1}\g=O(1).$ Then we have $\partial_r^{n-3}\W=O(1).$  When $\W\in C^{n-2}$, $\partial_r^{n-3}\W=O(r).$ So the left hand $\partial_r^{n-1}\g=O(r).$ It gives the necessary condition that $\g$ is in $C^n$ or has higher regularity.
\par In the following, we will calculate the formal power series of Weyl tensor in geodesic coordinates. $g=dr^2+g_r.$ From (\ref{3.6}),(\ref{3.7}) and \cite{G}, when n is odd,
\begin{equation}\lab{3.9}
g_r=h+g^{(2)}r^2+(even\ powers\ of\ r) + g^{(n-1)}r^{n-1}+g^{(n)}r^n+\cdots.
\end{equation}
When n is even,
\begin{equation}\lab{3.10}
g_r=h+g^{(2)}r^2+(even\ powers\ of\ r) + g^{(n)}r^{n}+fr^n\log r+\cdots,
\end{equation}
where $g^{(2i)}$ and $f$ are 2 tensors on $\partial M$ and $f$ is trace free and determined by $h$ locally. We are now considering the case when n is even. If $f=0,$ the n-th regularity of $g$ exists. Let $1\leq i,j\leq n.$ We know the Weyl tensor $W_{irjr}$ are
\begin{equation}\lab{3.11}
W_{irjr}=R_{irjr}-\frac{1}{n-1}(R_{ik}+R_{rr}g_{ij})+\frac{S}{n(n-1)}g_{ij}
\end{equation}
The formal power series of Weyl tensor contains $r$ and $\log r$ and we only need to check the coefficients of $r^{n-2}\log r.$
By a simple calculation, we get the coefficients of $r^{n-2}\log r$ of $R_{irjr}$ is
$$coeff(R_{irjr})=-\frac{n(n-1)}{2}f_{ij}.$$
And $$coeff(R_{ij})=-\frac{n(n-1)}{2}f_{ij},$$
$$coeff(R_{rr})=-\frac{n(n-1)}{2}h^{st}f_{st},$$
$$coeff(S)=-\frac{n(n-1)}{2}2h^{st}f_{st}.$$
So when $coeff(W_{irjr})=0,$ by (\ref{3.11}) we finally derive
\begin{equation}\lab{3.12}
f_{ij}-\frac{1}{n}(h^{st}f_{st})h_{ij}=0
\end{equation}
As $f$ is trace free, $f\equiv 0.$ This gives the n-th regularity of $g.$
From Theorem A in \cite{CDLS}, we know the log term is the only obstruction of the smoothness, then Remark 1.3 holds.
\section{The boundary condition}
\par In this section, we derive a boundary problem for $g$ and Ricci curvature of a conformal compact Einstein manifold in the harmonic coordinates as defined in section 2. We do it locally, that is, for any $p\in \M,$ there is a neighborhood $V$ contains $p$ and local atlas $\{x^\beta\}.$  Let $D=V\cap \M$ be the boundary portion. $g\in C^{m,\alpha}(V), W\in C^{m,\alpha}(V), h\in C^{m+2,\alpha}(D).$ We will give the Dirichlet and Neumann boundary conditions of $g$ and $Ric(g)$
on D.
\subsection{Dirichlet boundary conditions on $g_{ij}$}
\begin{equation}\lab{4.1}
g_{ij}=h_{ij}.
\end{equation}
\subsection{Dirichlet boundary conditions on $R_{ij}$}
\par We claim that
\begin{equation}\lab{4.2}
R_{ij}=\frac{n-1}{n-2}(Ric_h)_{ij}+(\frac{1}{2n}S-\frac{1}{2(n-2)}S_h)h_{ij}+\frac{n-1}{2n^2}H^2h_{ij}.
\end{equation}
Here $Ric_h$ and $S_h$ are Ricci curvature and scalar curvature of $(\M,h),$ H is the mean curvature, $H=g^{ij}A_{ij}.$ We use the following three lemmas to prove (\ref{4.2}).
\begin{lemm}\lab{lemm 4.1}
Let $\g=r^2g_+$ be a $C^2$ geodesic compactification of $(M,g_+)$ with boundary metric $h$ on $\M.$ Then on $\M,$
\begin{equation}\lab{4.3}
\S=\frac{n}{n-1}S_h.
\end{equation}

\begin{equation}\lab{4.4}
\R_{ij}=\frac{n-1}{n-2}(Ric_h)_{ij}-\frac{1}{2(n-1)(n-2)}S_hh_{ij}.
\end{equation}
\end{lemm}
\begin{proof}
Since we only need to study the Ricci curvature on $T\M,$ we can choose the coordinates $(r,x^1,\cdots,x^n)$ in $V$ where $(x^1,\cdots,x^n)$ are harmonic with respect to $h$ when restricted on $D.$ So $\g=dr^2+g_r.$ i.e.
$$g_{ri}=g^{ri}=0, g_{rr}=g^{rr}=1.$$
 Since the second fundamental form $\A=\D^2r$ vanishes on $\M.$ We have:
\begin{equation}\lab{4.5}
\begin{aligned}
\R_{ij}&=\g^{\alpha\beta}\R_{i\alpha\beta j}\\
       &=\g^{kl}((R_h)_{iklj}+\A_{il}\A_{kj}-\A_{ij}\A_{kl})+\R_{irrj}\\
       &=(R_h)_{ij}+\R_{irrj}.
\end{aligned}
\end{equation}
Taking trace for $i,j,$ we get
\begin{equation}\lab{4.6}
\R_{rr}=\frac{1}{2}(\S-S_h).
\end{equation}
For $\R_{irrj},$ we have:
\begin{equation}\lab{4.7}
\begin{aligned}
\R_{irrj}&=\g(\bar{\nabla}_{\partial_i}\bar{\nabla}_{\partial_r}\partial_r,\partial_j)-\g(\bar{\nabla}_{\partial_r}\bar{\nabla}_{\partial_i}\partial_r,\partial_j)-
              \g(\bar{\nabla}_{[\partial_r,\partial_i]}\partial_r,\partial_j)\\
         &=-\partial_r\g(\bar{\nabla}_{\partial_i}\partial_r,\partial_j)+\g(\bar{\nabla}_{\partial_i}\partial_r,\bar{\nabla}_{\partial_r}\partial_j)\\
         &=-\partial_r\A_{ij}.
\end{aligned}
\end{equation}
From (\ref{3.2}) and (\ref{3.3}) We have:
\begin{equation}\lab{4.8}
    \R_{ij}=-(n-1)\frac{\A_{ij}}{r}-\frac{\bar{\Delta}r}{r}\g_{ij}+o(1),
\end{equation}
\begin{equation}\lab{4.9}
    \S=-2n\frac{\bar{\Delta}r}{r}+o(1).
\end{equation}
So on $\M$ we have:
\begin{equation}\lab{4.10}
    \S=-2n\partial_r\bar{\Delta}r.
\end{equation}
In V, we have:
    $$\A_{ij}=-\frac{1}{n-1}(r\R_{ij}+\bar{\Delta}r\g_{ij})+o(r),$$
Taking trace on (4.8)
    $$\bar{\Delta}r=-\frac{1}{n}(r(\S-\R_{rr})+(n-1)\H)+o(r).$$
Then
\begin{equation}\lab{4.11}
    \partial_r\A_{ij}|_{r=0}=-\frac{1}{n-1}(\R_{ij}+\partial_r\bar{\Delta}r\g_{ij})|_{r=0},
\end{equation}
\begin{equation}\lab{4.12}
    \partial_r\bar{\Delta}r|_{r=0}=-\frac{1}{n}(\S-\R_{rr}+(n-1)\partial_r\H)|_{r=0}.
\end{equation}
At last, we only need to calculate $\partial_r\H|_{r=0}.$
\begin{equation}\lab{4.13}
\begin{aligned}
    \partial_r\H|_{r=0}&=\partial_r(\g^{ij}\A_{ij})|_{r=0}=(\partial_r\g^{ij})\A_{ij}|_{r=0}+\g^{ij}(\partial_r\A_{ij})|_{r=0}\\
                       &=0+\g^{ij}(-\R_{irrj})=-\R_{rr}.
\end{aligned}
\end{equation}
Combining the  formulas above ,we finally get (\ref{4.3}) and (\ref{4.4}).
\end{proof}
\begin{lemm}\lab{lemm 4.2}
Let $g=\rho^2g_+$ be a $C^{3,\lambda}$ conformal compactification and $\g=r^2g_+$ be the $C^{2,\lambda}$ geodesic conformal compactification of
$(M,g_+)$ with the same boundary metric $g|_{\M}=\g|_{\M}=h.$ If $r=u\rho, A=D^2\rho,$ then $A|_{\M}=-u_rh.$
\end{lemm}
\begin{proof}
In the coordinates $(r,x^1,x^2,\ldots,x^n),$ on $\M, 0=\A_{ij}=-\bar{\Gamma}_{ij}^r.$ So the connection $\nabla$ and $\bar{\nabla}$of $g$ and $\g$ have the relationship:
$$\Gamma_{ij}^r=\bar{\Gamma}_{ij}^r-\frac{1}{u}(\delta^r_ju_i+\delta^r_iu_j-g_{ij}u_r)=\frac{1}{u}u_rh_{ij}.$$
As $g=u^{-2}\g, grad_g=u^2 grad_{\g}.$
\begin{equation}\lab{4.14}
\begin{aligned}
    A_{ij}&=D^2\rho(\partial_i,\partial_j)=g(\nabla_{\partial_i}\nabla\rho,\partial_j)=-g(\nabla\rho,\nabla_{\partial_i}\partial_j)\\
          &=-\Gamma_{ij}^rg(\nabla\rho,\partial_r)=-\Gamma_{ij}^r\g(\bar{\nabla}\rho,\partial_r)\\
          &=-\Gamma_{ij}^r\g(\bar{\nabla}(\frac{r}{u}),\partial_r)=-\Gamma_{ij}^r\g(\frac{u\bar{\nabla}r-r\bar{\nabla}u}{u^2},\partial_r)\\
          &=-\Gamma_{ij}^r\g(\bar{\nabla}r,\bar{\nabla}r)=-u_rh_{ij}
\end{aligned}
\end{equation}
\end{proof}
Lemma \ref{lemm 4.2} tells us $u_r=-\frac{H}{n}.$ Since $u|_{\M}\equiv 1,$ we get
$$\bar{\nabla}u=-\frac{H}{n}\bar{\nabla}r.$$
\begin{lemm}\lab{lemm 4.3}
$g,\g$ are defined as in Lemma \ref{lemm 4.2}. Then on $\M,$
$$R_{ir}=\frac{n-1}{n}\frac{\partial H}{\partial x_i},$$
$$R_{rr}=\frac{1}{2}(S-S_h)-\frac{n-1}{2n}H^2,$$
\begin{equation}\lab{4.15}
R_{ij}=\R_{ij}+(\frac{1}{2n}(S-\S))h_{ij}+\frac{n-1}{2n^2}H^2h_{ij}.
\end{equation}
\end{lemm}
\begin{proof}
By the standard formulas for conformal changes of the metric $g=u^{-2}\g,$ the Ricci curvature of $g$ and $\g$ are related by
$$Ric=\R ic+(n-1)\frac{\bar{D}^2u}{u}+(\frac{\bar{\Delta}u}{u}+\frac{n|\bar{\nabla}u|^2_{\g}}{u^2})\g.$$
Since
$$\bar{\Delta}u=div\bar{\nabla}u=div(-\frac{H}{n}\bar{\nabla}r)=-\frac{\partial_rH}{n},$$
$$\D^2u(\partial_i,\partial_j)=0,$$
$$\D^2u(\partial_i,\partial_r)=-u_{ir}=\frac{1}{n}\frac{\partial H}{\partial x_i},$$
$$\D^2u(\partial_r,\partial_r)=-\frac{\partial_rH}{n}=\bar{\Delta}u.$$
On $\M,$ we have
$$R_{ir}=\R_{ir}+\frac{n-1}{n}\frac{\partial H}{\partial x_i}=\frac{n-1}{n}\frac{\partial H}{\partial x_i},$$
$$R_{rr}=\R_{rr}+n\bar{\Delta}u+\frac{H^2}{n},$$
\begin{equation}\lab{4.16}
R_{ij}=\R_{ij}+(\bar{\Delta}u+\frac{H^2}{n})h_{ij}.
\end{equation}
So
$$S=\S+2n\bar{\Delta}u+\frac{n+1}{n}H^2.$$
Which implies
\begin{equation}\lab{4.17}
    \bar{\Delta}u=\frac{1}{2n}(S-\S-\frac{n+1}{n}H^2).
\end{equation}
From (\ref{4.16}) and (\ref{4.17}), the lemma \ref{lemm 4.3} holds.
\end{proof}
At last, the lemma \ref{lemm 4.1} and lemma \ref{lemm 4.3} implies (\ref{4.2}).
\subsection{Neumann boundary conditions on $g^{0\alpha}$}
\par In this section, we use the harmonic coordinates $\{x^\beta\}_{\beta=0}^n$ as defined in section 2. Let $N=\frac{\nabla x_0}{|\nabla x_0|}$ be the unit norm vector on $\M.$ In coordinates $\{x^\beta\}_{\beta=0}^n,$
$$N=(g^{00})^{-\frac{1}{2}}g^{0\beta}\partial_\beta.$$
Then these are of the form
\begin{equation}\lab{4.18}
    N(g^{00})=-2Hg^{00}.
\end{equation}
\begin{equation}\lab{4.19}
    N(g^{0i})=-Hg^{0i}+\frac{1}{2}(g^{00})^{-\frac{1}{2}}g^{i\beta}\partial_\beta g^{00}.
\end{equation}
\begin{proof}
Let $\{e_i\}_{i=1}^n$ be the orthonormal basis at a given point $p\in\M.$ So we have:
\begin{equation}\lab{4.20}
    0=\Delta x^\alpha=div \nabla x^\alpha=\sum\limits_{i=1}^ng(\nabla_{e^i}\nabla x^\alpha,e_i)+g(\nabla_N\nabla x^\alpha,N)
\end{equation}
Write $\nabla x^\alpha=(\nabla x^\alpha)^T+(\nabla x^\alpha)^N,$ where
$$(\nabla x^\alpha)^T=\nabla_hx^\alpha,(\nabla x^\alpha)^N=g(\nabla x^\alpha,N)N=(g^{00})^{-\frac{1}{2}}g^{0\alpha}N.$$
Since $$0=\Delta_h x^j=div \nabla_h x^j=\sum\limits_{i=1}^nh(\nabla^h_{e^i}(\nabla x^j)^T,e_i)=\sum\limits_{i=1}^ng(\nabla_{e^i}(\nabla x^j)^T,e_i)$$
It also holds for $j=0$ as $x^0\equiv 0$ on $\M.$ Then (\ref{4.20}) turns into
\begin{equation}\lab{4.21}
\sum\limits_{i=1}^ng(\nabla_{e^i}((g^{00})^{-\frac{1}{2}}g^{0\alpha} N),e_i)+g(\nabla_N\nabla x^\alpha,N)=0
\end{equation}
The first term is just $(g^{00})^{-\frac{1}{2}}g^{0\alpha}\sum\limits_{i=1}^nA(e_i,e_i)=H(g^{00})^{-\frac{1}{2}}g^{0\alpha}$ and the second term is
\begin{equation}\lab{4.22}
\begin{aligned}
g(\nabla_N\nabla x^\alpha,N)&=\frac{1}{|\nabla x^0|^2}g(\nabla_{\nabla x^0}\nabla x^\alpha,\nabla x^0)\\
                     &=\frac{1}{|\nabla x^0|^2}[\nabla x^0 (g(\nabla x^\alpha,\nabla x^0))-g(\nabla x^\alpha,\nabla_{\nabla x^0 }\nabla x^0)]\\
                     &=\frac{1}{|\nabla x^0|^2}[\nabla x^0(g^{0\alpha})-g(\nabla x^0,\nabla_{\nabla x^\alpha}\nabla x^0)]\\
                     &=\frac{1}{|\nabla x^0|}N(g^{0\alpha})-\frac{1}{2|\nabla x^0|^2}\nabla x^\alpha(g^{00})
\end{aligned}
\end{equation}
We know $|\nabla x^0|=(g^{00})^{\frac{1}{2}},$ then (4.21) is just
\begin{equation}\lab{4.23}
Hg^{0\alpha}+N(g^{0\alpha})-\frac{1}{2}(g^{00})^{-\frac{1}{2}}\nabla x^\alpha(g^{00})=0
\end{equation}
So when $\alpha=0,$ we have $N(g^{00})=-2Hg^{00}.$ When $\alpha=i,$ we have
 $$N(g^{0i})=-Hg^{0i}+\frac{1}{2}(g^{00})^{-\frac{1}{2}}g^{i\beta}\partial_\beta g^{00}.$$
\end{proof}
\subsection{Dirichlet boundary conditions on $R_{\alpha\beta}$}
In section 4.2, we already know the formulas of $R_{ij}$ on $\M$ and the mixed components $R_{ri}$ and $R_{rr}$ of Ricci Curvature in the coordinates $(r,x^1,x^2,\cdots,x^n)$. That is,
$$R_{ir}=\frac{n-1}{n}\frac{\partial H}{\partial x_i},$$
$$R_{rr}=\frac{1}{2}(S-S_h)-\frac{n-1}{2n}H^2.$$
\par Now we study the  the mixed components $R_{0i}$ and $R_{00}$ of Ricci Curvature in the harmonic coordinates $(x^0,x^1,\cdots,x^n).$ In fact, as the vector $N=\frac{\nabla x_0}{|\nabla x_0|}$ is also the unit norm vector on $\M$ with respect to $g,$ we have $N=\bar{\nabla}r=\nabla r.$
i.e. $\nabla r=(g^{00})^{-\frac{1}{2}}g^{0\beta}\partial_\beta.$ Then
\begin{equation}\lab{4.24}
R_{0i}=(g^{00})^{-\frac{1}{2}}\frac{n-1}{n}\frac{\partial H}{\partial x_i}-\frac{g^{0j}}{g^{00}}R_{ij}
\end{equation}
\begin{equation}\lab{4.25}
R_{00}=\frac{1}{(g^{00})^2}(g^{0i}g^{0j}R_{ij}+g^{00}(\frac{1}{2}(S-S_h)-\frac{n-1}{2n}H^2))
\end{equation}
\subsection{Neumann boundary conditions on $R_{0i}$}
The Dirichlet condition for $R_{0i}$ in (\ref{4.24} is not good because there are second order differential terms of metric in the right side. Now we consider the differential terms of Ricci curvature.
\par Since the scalar curvature is a constant, by the second Bianchi identity, we have
$$0=\frac{1}{2}S_{,\alpha}=g^{\eta\beta}R_{\eta\alpha,\beta}=g^{\eta\beta}\partial_\beta R_{\alpha\eta}-g^{\eta\beta}\Gamma_{\alpha\beta}^\tau R_{\eta\tau}.$$
Then
$$g^{0\beta}\partial_\beta R_{0\alpha}=-g^{j\beta}\partial_\beta R_{j\alpha}+g^{\eta\beta}\Gamma_{\alpha\beta}^\tau R_{\eta\tau}.$$
Let $\alpha=i,$ we get
\begin{equation}\lab{4.26}
N(R_{0i})=(g^{00})^{-\frac{1}{2}}(-g^{j\beta}\partial_\beta R_{ji}+g^{\eta\beta}\Gamma_{i\beta}^\tau R_{\eta\tau})
\end{equation}
\section{Proof of the main theorem}
In this section, we prove the main theorem. Suppose $m\geq 3$ and $\alpha\in(0,1).$ For any point $p\in \M,$ choose the harmonic coordinates $\{x^\beta\}_{\beta=0}^n$ in its neighborhood $V.$ Let $D=V\cap\M$ be the boundary portion. Now we have $g\in C^{m,\alpha}(V), h\in C^{m+2,\alpha}(D),
W\in C^{m,\alpha}(V).$
\subsection{Regularity of the metric}
\par \textbf{Step 1}: regularity of the Ricci curvature.\\
\par We begin with (\ref{2.5}) and the right side of (\ref{2.5}) are in $C^{m-2,\alpha}.$
\par On $\M$, we already derive the formulas of $R_{\alpha\beta}$ in section 3. As $H=g^{ij}A_{ij},$ and $A_{ij}=\frac{1}{2}(g^{00})^{\frac{1}{2}}g^{0\beta}(\partial_\beta g_{ij}-\partial_i g_{\beta j}-\partial_j g_{\beta i})\in C^{m-1,\alpha}(D).$ Then the Dirichlet condition of $R_{ij}$ and $R_{00}$ given by  (\ref{4.4} and (\ref{4.25} shows that
$$R_{ij}\in C^{m-1,\alpha}(D),R_{00}\in C^{m-1,\alpha}(D).$$
By standard elliptic regularity theory,
 $$R_{ij}\in C^{m-1,\alpha}(V),R_{00}\in C^{m-1,\alpha}(V).$$
 Then by the Neumann boundary conditions on $R_{0i}$ given by (\ref{4.26}), we also have $R_{0i}\in C^{m-1,\alpha}(V)$ since $N(R_{0i})\in C^{m-2,\alpha}(D).$
\\
In the following, we prove that $g_{\alpha\beta}\in C^{m+1,\alpha}(V)$ when $R_{\alpha\beta}\in C^{m-1,\alpha}(V).$ If it holds, we can get that $g_{\alpha\beta}\in C^{m+2,\alpha}(V)$ by repeating the steps.
\par \textbf{Step 2}: regularity of $g_{ij}.$
\\
\par In harmonic coordinates, we have
$$\Delta g_{ij}=-2R_{ij}+Q(g,\partial g)$$
Here $Q$ is a $1^{st}$ term of $g.$ So $\Delta g_{ij}\in C^{m-1,\alpha}(V),$ together with the boundary condition $g_{ij}=h_{ij}\in C^{m+2,\alpha}(D).$
We get $g_{ij}\in C^{m+1,\alpha}(V).$
\\
\par \textbf{Step 3}: regularity of $g_{0\beta}.$
\\
\par In section 4, we obtain the Neumann boundary condition of $g^{0\beta}$ which contain $H.$ Sense $H\in C^{m-1,\alpha},$ we can't improve the regularity of $g^{0\beta}$ in this condition. Now we are going to calculate the oblique derivative of $g_{0\beta}$ on $\M.$
As $g^{0\alpha}g_{\alpha\beta}=\delta_\beta^0,$
\begin{equation}\lab{5.1}
\begin{aligned}
0&=N(g^{0\alpha}g_{\alpha\beta})=N(g^{00}g_{0\beta})+N(g^{0j}g_{j\beta})\\
&=g^{00}N(g_{0\beta})+g_{0\beta}N(g^{00})+g_{j\beta}N(g^{0j})+g^{0j}N(g_{j\beta})\\
&=g^{00}N(g_{0\beta})+g_{0\beta}(-2Hg^{00})+g_{j\beta}(\frac{1}{2}(g^{00})^{-\frac{1}{2}}g^{j\tau}\partial_\tau g^{00}-Hg^{0j})+g^{0j}N(g_{j\beta})\\
&=g^{00}N(g_{0\beta})-2Hg_{0\beta}g^{00}+\frac{1}{2}(g^{00})^{-\frac{1}{2}}(\delta_\beta^\tau-g_{0\beta}g^{0\tau})\partial_\tau g^{00}-H(-g_{0\beta}g^{00})+g^{0j}N(g_{j\beta})\\
&=g^{00}N(g_{0\beta})+\frac{1}{2}(g^{00})^{-\frac{1}{2}}\partial_\beta g^{00}+g^{0j}N(g_{j\beta})-H\delta_\beta^0
\end{aligned}
\end{equation}
When $\beta=0,$ we have
\begin{equation}\lab{5.2}
g^{00}N(g_{00})+\frac{1}{2}(g^{00})^{-\frac{1}{2}}\partial_0 g^{00}+g^{0j}N(g_{0j})-H=0
\end{equation}
When $\beta=i,$ we have
\begin{equation}\lab{5.3}
g^{00}N(g_{0i})+\frac{1}{2}(g^{00})^{-\frac{1}{2}}\partial_i g^{00}+g^{0j}N(g_{ij})=0
\end{equation}
Now we consider the elliptic system of $g^{00},g_{01},g_{02},\cdots,g_{0n}$:
 \begin{equation}\lab{5.4}
 \left\{
    \begin{array}{l}
    \Delta g^{00}=Q(g,\partial g,Ric)
    \\ \Delta g_{01}=-\frac{1}{2}R_{01}+Q(g,\partial g)
    \\ \ \ \vdots
    \\\Delta g_{0n}=-\frac{1}{2}R_{0n}+Q(g,\partial g)
    \end{array}
 \right.
 \end{equation}
And from (\ref{4.18}), (\ref{5.3}) and the expression of $H,$ the regularities of $g_{ij}$ we obtain the boundary condition:
 \begin{equation}\lab{5.5}
 \left\{
    \begin{array}{l}
     N(g^{00})-2Pg^{ij}\partial_i g_{0j}\in C^{m,\alpha}(D)
    \\ N(g_{01})+\frac{1}{2P}\partial_1g^{00}\in C^{m,\alpha}(D)
    \\ \ \ \vdots
    \\N(g_{0n})+\frac{1}{2P}\partial_ng^{00}\in C^{m,\alpha}(D)
    \end{array}
 \right.
 \end{equation}
Where $P=(g^{00}(x))^{\frac{3}{2}}.$
We are going to prove the lemma:
\begin{lemm}\lab{lemm 5.1}
Let $u^0=g^{00}, u^i=g_{0i}, i=1,2,\cdots,n.$ Then (\ref{5.4}) has the form
$$L_{\alpha\beta}u^\beta(x)=f_\alpha(x), \ \ x\in V. $$
The boundary condition (\ref{5.5}) has the form
$$B_{\alpha\beta}u^\beta(x)=g_\alpha(x), \ \ x\in D.$$
Then the operator $L$ is proper elliptic and the boundary operator $B$ satisfies the complementing condition with respect to the system $(L,B).$
\end{lemm}
\begin{proof}
For any $n+1$ vector $\xi=(\xi_0,\xi_1,\cdots,\xi_n),$ we consider the principal part of L
\begin{equation}\lab{5.6}
L'_{\alpha\beta}(x,\xi)=\left[
\begin{matrix}
g^{\alpha\beta}\xi_\alpha\xi_\beta&0&\cdots&0&\\
0&g^{\alpha\beta}\xi_\alpha\xi_\beta&\cdots&0&\\
\vdots&\vdots&\ddots&\vdots&\\
0&0&\cdots&g^{\alpha\beta}\xi_\alpha\xi_\beta&
\end{matrix}
\right]
\end{equation}
Then$$L_{\alpha\beta}(x,\xi)=\det(L'_{\alpha\beta}(x,\xi))=|\xi|_g^{2(n+1)}.$$
For any $\xi\neq 0, L_{\alpha\beta}(x,\xi)\neq 0,$ so L is elliptic.
\\
For each pair of $\xi$ and $\xi'$ of linearly independent vectors, the equation
$$L_{\alpha\beta}(x,\xi+z\xi')=0$$
is equivalent to
$$(z^2\cdot|\xi'|_g^2+2<\xi,\xi'>_gz+|\xi|_g^2)^{n+1}=0.$$
It has $n+1$ roots with positive imaginary part and $n+1$ with negative imaginary part. So $L$ is proper elliptic.
\\ \\
 For any $x_0\in D,$ let $n=(1,0,\cdots,0)$ denote the unit normal at $x_0$ and $\xi=(0,\xi_1,\cdots,\xi_n)$ denote any nonzero real
vector tangent to $D$ at $x_0.$ Let $z_s^+(x_0,\xi),s=0,1,\cdot,n$ be the roots of $L_{\alpha\beta}(x_0,\xi+zn)=0$ with positive imaginary.
\\Define $$L_0^+(x_0,\xi;z)=\prod_{s=0}^n(z-z_s^+(x_0,\xi).$$
Let $L^{\alpha\beta}(x_0,\xi+zn)$ be matrix adjoint to $L'_{\alpha\beta}(x_0,\xi+zn).$ Now we define
$$Q_{r\beta}=B'_{r\alpha}(x_0,\xi+zn)\cdot L^{\alpha\beta}(x_0,\xi+zn)$$
as polynomials in z, where $B'_{r\alpha}$ is the principal part of $B.$
\\ Then  $B$ satisfies the complementing condition with respect to the system $(L,B)$ if and only if the rows
of the $Q$ matrix are linearly independent modulo $L_0^+(x_0,\xi;z),$ that is,
the polynomial
$$\sum_{r=0}^nC_rQ_{r\beta}(x_0,\xi;z)\equiv 0 \ \ (mod L_0^+)$$
only if $C_r$ are all 0.
\\
By a simple calculation, we  can get $z_s=\sqrt{-1}\frac{|\xi|_h}{\sqrt{g^{00}}}$ for $s=0,1,\cdots,n.$ Then
$$L_0^+(x_0,\xi;z)=(z-\sqrt{-1}\frac{|\xi|_h}{\sqrt{g^{00}}})^{n+1}.$$
From the above, we know $$L'_{\alpha\beta}(x_0,\xi+zn)=|\xi+zn|_g^2\cdot\delta_{\alpha\beta}=(z^2g^{00}+|\xi|_h^2)\cdot\delta_{\alpha\beta}.$$
Its adjoint matrix is
$$L^{\alpha\beta}(x_0,\xi+zn)=(z^2g^{00}+|\xi|_h^2)^n\cdot\delta^{\alpha\beta}.$$
The principal part of B is
\begin{equation}\lab{5.7}
B'_{\alpha\beta}(x,\xi)=\left[
\begin{matrix}
z&-2Pg^{i1}\xi_i&-2Pg^{i2}\xi_i&\cdots&-2Pg^{in}\xi_i&\\
\frac{1}{2P}\xi_1&z&0&\cdots&0&\\
\frac{1}{2P}\xi_2&0&z&\cdots&0&\\
\vdots&\vdots&\vdots&\ddots&\vdots&\\
\frac{1}{2P}\xi_n&0&0&\cdots&z&\\
\end{matrix}
\right]
\end{equation}
Then
\begin{equation}\lab{5.8}
\begin{aligned}
\sum_{r=0}^nC_rQ_{r\beta}&=\sum_{r=0}^nC_rB'_{r\alpha}\cdot L^{\alpha\beta}\\
                         &=\sum_{r=0}^nC_rB'_{r\alpha}\cdot(z^2g^{00}+|\xi|_h^2)^n\cdot\delta^{\alpha\beta}\\
                         &=\sum_{r=0}^nC_rB'_{r\beta}\cdot(z^2g^{00}+|\xi|_h^2)^n\\
                         &\equiv 0\ \ (mod (z-\sqrt{-1}\frac{|\xi|_h}{\sqrt{g^{00}}})^{n+1})
\end{aligned}
\end{equation}
Which implies
\begin{equation}\lab{5.9}
z-\sqrt{-1}\frac{|\xi|_h}{\sqrt{g^{00}}}\mid\sum_{r=0}^nC_rB'_{r\beta}
\end{equation}
for any $0\leq\beta\leq n.$\\
When $\beta\geq 1,(\ref{5.9})\Rightarrow z-\sqrt{-1}\frac{|\xi|_h}{\sqrt{g^{00}}}\mid \frac{C_0}{2P}\xi_\beta+C_\beta z.$
Then
\begin{equation}\lab{5.10}
C_\beta =-\frac{C_0\sqrt{g^{00}}\xi_\beta}{2P|\xi|_h\cdot\sqrt{-1}}
\end{equation}
When $\beta=0,$
$$(\ref{5.9})\Rightarrow z-\sqrt{-1}\frac{|\xi|_h}{\sqrt{g^{00}}}\mid C_0z-2PC_1g^{i1}\xi_i-2PC_2g^{i2}\xi_i-\cdots-2PC_ng^{in}\xi_i$$
With (\ref{5.10}), we have
\begin{equation}\lab{5.11}
z-\sqrt{-1}\frac{|\xi|_h}{\sqrt{g^{00}}}\mid C_0z+\frac{C_0\sqrt{g^{00}}}{2P|\xi|_h\cdot\sqrt{-1}}
\end{equation}
By a linear transformation, we can make $g^{00}(x_0)\neq 1, \forall x_0\in D.$ Then (\ref{5.11}) shows that $C_0=0,$ and (\ref{5.10}) implies
$C_r=0,r=1,2,\cdots,n.$

\end{proof}
Lemma \ref{lemm 5.1} and theorem 6.3.7 in \cite{Mo} tell us that $g^{00}, g_{01},\cdots,g_{0n}$ are all in $C^{m+1,\alpha}.$ Then the boundary condition
(\ref{5.2}) can be written as
$$g^{00}N(g_{00})=Q(g,\partial g^{00},\partial g_{0i},\partial g_{ij})\in C^{m,\alpha}(D)$$
With the elliptic equation $$\Delta g_{00}=-2R_{00}+Q(g,\partial g)\in C^{m-1,\alpha}(V),$$
we finally derive $g_{00}\in C^{m+1,\alpha}.$

\par Now we know that $g\in C^{m+1,\alpha}.$ Back to step 1, we have $A_{ij}\in C^{m,\alpha}(D),$ then $R_{\alpha\beta}\in C^{m,\alpha}(V).$ Repeating the steps above, we can get $g_{\alpha\beta}\in C^{m+2,\alpha}(V).$ Then we complete the proof.
\subsection{Regularity of the structure and the defining function}
We have already proved that $g$ is in $ C^{m+2,\alpha}$ in structure $\{x^\beta\}.$ It is trivial that $\{x^\beta\}$ is a $ C^{m+3,\alpha}$ structure of $\overline{M}.$
\par In section 2, when we make constant scalar compactification, we obtain that $u$ is in $C^{m,\alpha}(y).$ When we change the y-coordinates to harmonic coordinates x, we know that $x^\beta$ are $C^{m+1,\alpha}$ functions of $y^{\beta}.$ So the defining function $\rho\in C^{m,\alpha}(x).$\par Since the initial compactification $g$ is smooth in $M$ and the initial defining function is smooth in y-coordinates, then $\rho\in C^\infty(x)$ in $M.$
\par For any $p\in\M,$ consider the neighborhood $V$ of $p$ and $D=\M\cap V.$ By a linear transformation, we can assume that at $p,$ $g_{\alpha\alpha}=1, g_{ij}=g_{02}=g_{03}=\cdots=g_{0n}=0 (i\neq j), g_{01}=\delta$ for some $\delta>0$ satisefying that $1-\delta$ is a very small posivive number.
When $g_+$ is Einstein, from (\ref{3.2}) and (\ref{3.3}), we have
$$Ric-\frac{Sg}{n+1}=-(n-1)\frac{D^2\rho}{\rho}+\frac{n-1}{n+1}\frac{\Delta\rho}{\rho}g.$$
In local coordinates, when acting on $(\frac{\partial}{\partial x^0},\frac{\partial}{\partial x^1}),$ we have
\begin{equation}\lab{5.12}
\Delta\rho-(n+1)\cdot g_{01}^{-1}\cdot D^2\rho(\frac{\partial}{\partial x^0},\frac{\partial}{\partial x^1})=\frac{n+1}{n-1}\cdot g_{01}^{-1}\cdot\rho(Ric_{01}-\frac{Sg_{01}}{n+1})
\end{equation}
When $1-\delta$ is small enough, we can find that the left hand of (\ref{5.12}) is a uniformly elliptic operator on $\rho$ locally. As $\rho R_{01}\in C^{m,\alpha}(\overline{M}),$ $\rho|_D\equiv 0,$ so in $V,$ $\rho\in C^{m+2,\alpha}(x).$
\\
\par To prove Remark(\ref{rema1.2}), we only need to show that $\rho R_{01}\in C^{m+1,\alpha}(\overline{M}).$ When $\dim M=n+1$ is even, we can define the obstruction tensor $\mathcal{O}_{ij}.$ In local coordinates:
\begin{equation}\lab{5.13}
\mathcal{O}_{ij}=\Delta^{\frac{n+1}{2}-2}(P_{ij,k}^{\ \ \ \ k}-P_{ik,j}^{\ \ \ \ k})+Q_n.
\end{equation}
Where the $P_{ij}$ is defined by $P_{ij}=\frac{1}{2}R_{ij}-\frac{S}{12}g_{ij}$ and $Q_n$ denotes quadratic
and higher terms in metric involving at most n-th derivatives. $\mathcal{O}_{ij}$ is conformally invariant of weight $2-n$ and if $g_{ij}$ is conformal to an Einstein metric, then $\mathcal{O}_{ij}=0$ (see more in \cite{GH}).
\\
Since the scalar curvature of $(\overline{M},g)$ is constant, (\ref{5.13}) can be written in the following form:
\begin{equation}\lab{5.14}
\Delta^{\frac{n+1}{2}-1}R_{ij}=Q_n.
\end{equation}
Now we consider the function $\rho R_{01}.$ Through a direct calculation, $\Delta(\rho R_{ij})=\rho\Delta R_{ij}+Q_3^2.$ Here $Q_3^2$ denotes quadratic
and higher terms in metric involving at most 3-th derivatives and in $\rho$ involving at most 2-th derivatives.
\\ Then we use iterative method to obtain that $\Delta^k(\rho R_{ij})=\rho\Delta^k R_{ij}+Q_{2k+1}^{2k}$ for $1\leq k\leq\frac{n+1}{2}-1.$
Let $k=\frac{n+1}{2}-1,$ we have an elliptic equation of second order with Dirichlet boundary condition:
\begin{equation}\lab{5.15}
 \left\{
    \begin{array}{l}
    \Delta(\Delta^{\frac{n+1}{2}-2}(\rho R_{01}))=Q_n^{n-1} \ in \  \overline{M}
    \\ \Delta^{\frac{n+1}{2}-2}(\rho R_{01})|_{\M}=Q_{n-2}^{n-3}
    \end{array}
 \right.
 \end{equation}
 Since $g$ and $\rho$ are all in $C^{m+2,\alpha},m+2\geq n,$ we have
 $$\Delta^{\frac{n+1}{2}-2}(\rho R_{01})\in C^{m+2-(n-2),\alpha}(\overline{M}).$$
 Then we consider the equation
 \begin{equation}\lab{5.16}
 \left\{
    \begin{array}{l}
    \Delta(\Delta^{\frac{n+1}{2}-3}(\rho R_{01}))\in C^{m+2-(n-2),\alpha}(\overline{M})
    \\ \Delta^{\frac{n+1}{2}-3}(\rho R_{01})|_{\M}=Q_{n-4}^{n-5}\in C^{m+2-(n-4),\alpha}(\M)
    \end{array}
 \right.
 \end{equation}
 So we have $\Delta^{\frac{n+1}{2}-3}(\rho R_{01})\in C^{m+2-(n-4),\alpha}(\overline{M})$
 \\
 Keep using the equation, we finally get $\rho R_{01}\in C^{m+1,\alpha}(\overline{M}),$ Which implies $\rho\in C^{m+3,\alpha}(\overline{M})$ by (\ref{5.12}). $\{x^\beta\}$ is a $C^{m+3,\alpha}$ structure of $\overline{M},$ so $C^{m+3,\alpha}$ regularity of $\rho$ is the best result we can get.

\noindent{Xiaoshang Jin}\\
  Department of Mathematics, Nanjing University, Nanjing, 210093, P.R. China.
 \\Email address:{dg1521006@smail.nju.edu.cn}

\end{document}